\def\draftdate{December 14, 2021}

\documentclass{amsart}
\usepackage{amssymb}
\usepackage{xy}

\newcommand{\HFT}{\Fib(\tau)}
\newcommand{\HFC}{\Fib(\kappa)}
\newcommand{\HFf}{\Fib(f)}

\newcommand{\NR}{\aO}

\newcommand{\et}{\textup{\'et}}

\let\iso\cong
\let\sma\wedge
\newcommand{\smaL}{\sma}

\newcommand{\mystrut}{\vbox to 1em{\vss}}

\renewcommand{\to}{\mathchoice{\longrightarrow}{\rightarrow}{\rightarrow}{\rightarrow}}
\newcommand{\from}{\mathchoice{\longleftarrow}{\leftarrow}{\leftarrow}{\leftarrow}}

\newcommand{\overto}[1]{\xrightarrow{\,#1\,}}

\newcommand{\phat}{^{\scriptscriptstyle\wedge}_{p}}

\let\catsymbfont\mathcal

\newcommand{\aO}{{\catsymbfont{O}}}

\newcommand{\bC}{{\mathbb{C}}}

\newcommand{\bF}{{\mathbb{F}}}
\newcommand{\bFl}{{\mathbb{F}_{l}}}
\newcommand{\bbFl}{{\bar{\mathbb{F}}}_{l}}
\newcommand{\bbFlt}{{\bar{\mathbb{F}}}_{l}^{\times}}

\newcommand{\bS}{{\mathbb{S}}}
\newcommand{\bZ}{{\mathbb{Z}}}

\newcommand{\bZp}{{\mathbb{Z}_{p}}}
\newcommand{\bZpi}{\bQp/\bZp}
\newcommand{\bZpt}{{\mathbb{Z}^{\times}_{p}}}
\newcommand{\bQ}{{\mathbb{Q}}}

\newcommand{\bQS}{{\mathbb{Q}_{S}}}
\newcommand{\bQp}{{\mathbb{Q}_{p}}}
\newcommand{\bbQp}{{\bar{\mathbb{Q}}_{p}}}
\newcommand{\bQpab}{{\bar{\mathbb{Q}}_{p}^{\ab}}}
\newcommand{\bQpnr}{({\mathbb{Q}_{p}})^{\nr}}
\newcommand{\bQpt}{{\mathbb{Q}^{\times}_{p}}}

\newcommand{\oL}{\mathcal{L}}

\def\quickop#1{\expandafter\DeclareMathOperator\csname
#1\endcsname{#1}}
\quickop{id}\quickop{Id}
\quickop{Tor}\quickop{Ext}
\quickop{holim}\quickop{hocolim}\quickop{hofib}
\quickop{Tel}
\quickop{Pic}\quickop{Gal}
\quickop{Spec}
\quickop{trc}
\quickop{spec}
\quickop{colim}
\quickop{sd}
\quickop{dom}
\quickop{cod}
\quickop{Hom}
\quickop{Sp}
\quickop{op}
\quickop{sk}
\quickop{cosk}
\quickop{disc}
\quickop{Tot}
\quickop{Mod}
\quickop{Fib}
\quickop{Cof}
\quickop{Diff}
\quickop{ab}
\quickop{nr}
\quickop{coker}
\quickop{rec}\quickop{red}

\numberwithin{equation}{section}

\newtheorem{thm}[equation]{Theorem}

\newtheorem{cor}[equation]{Corollary}

\newtheorem{prop}[equation]{Proposition}
\theoremstyle{definition}
\newtheorem{defn}[equation]{Definition}

\theoremstyle{remark}
\newtheorem{rem}[equation]{Remark}

\ifx\texorpdfstring\undefined\def\texorpdfstring#1#2{#1}\fi

\xyoption{arrow}
\xyoption{matrix}
\xyoption{cmtip}
\SelectTips{cm}{}

\newcommand{\term}[1]{\textit{#1}}

\newdir{ >}{{}*!/-5pt/\dir{>}}

\begin{document}

\title%
[The eigensplitting of the fiber of the cyclotomic trace]
{The eigensplitting of the fiber of the cyclotomic trace for the
sphere spectrum}

\author{Andrew J. Blumberg}
\address{Department of Mathematics, Columbia University,
New York, NY \ 10027}
\email{blumberg@math.columbia.edu}
\thanks{The first author was supported in part by NSF grants DMS-1812064, DMS-2104420}
\author{Michael A. Mandell}
\address{Department of Mathematics, Indiana University,
Bloomington, IN \ 47405}
\thanks{The second author was supported in part by NSF grants DMS-1811820, DMS-2104348}
\email{mmandell@indiana.edu}

\date{\draftdate} 

\begin{abstract}
Let $p\in \mathbb Z$ be an odd prime.  We show that the fiber sequence
for the cyclotomic trace of the sphere spectrum $\mathbb S$ admits an
``eigensplitting'' that generalizes known splittings on $K$-theory and
$TC$.  We identify the summands in the fiber as the covers of $\mathbb
Z_{p}$-Anderson duals of summands in the $K(1)$-localized
algebraic $K$-theory of $\mathbb Z$.  Analogous results hold for the
ring $\mathbb Z$ where we prove that the $K(1)$-localized fiber
sequence is self-dual for $\mathbb Z_{p}$-Anderson duality,
with the duality permuting the summands by $i\mapsto p-i$ (indexed mod
$p-1$).  We explain an intrinsic characterization of the summand we
call $Z$ in the splitting $TC(\mathbb Z)^{\wedge}_{p}\simeq j \vee \Sigma
j'\vee Z$ in terms of units in the $p$-cyclotomic tower of
$\mathbb Q_{p}$.
\end{abstract}

\maketitle

\section{Introduction}

The algebraic $K$-theory of the sphere spectrum, $K(\bS)$, is an
object of basic and fundamental interest, relating geometric topology
and arithmetic. Celebrated work of Waldhausen establishes a
comparison between $K(\bS)$ and a stable space of $h$-cobordisms for the disk
$D^n$.  On the other hand, $K(\bS)$ is intimately related to 
$K(\bZ)$, the algebraic $K$-theory of the integers, which encodes
arithmetic invariants (e.g., Bernoulli numerators and denominators).
For instance, the natural map $K(\bS) \to K(\bZ)$ is a rational
equivalence, and the latter is understood rationally by old work of Borel.

At a prime $p$, our understanding of algebraic $K$-theory of ring
spectra relies on trace methods.  B\"okstedt, Hsiang, and Madsen
constructed a topological version of negative cyclic homology called
topological cyclic homology ($TC$) and a Chern character $K \to TC$,
the cyclotomic trace.  Following earlier work of Rognes, in a previous
paper we studied the homotopy groups of $K(\bS)$ in terms of the
cyclotomic trace and linearization maps: a basic theorem of Dundas
(building on work of Goodwillie and McCarthy) provides a homotopy
cartesian square 
\begin{equation*}\label{eq:dgmsquare}
\begin{gathered}
\xymatrix{
K(\bS)\phat \ar[r]\ar[d] & K(\bZ)\phat \ar[d] \\ TC(\bS)\phat \ar[r] & TC(\bZ)\phat, }
\end{gathered}
\end{equation*}
where the maps $K(\bS) \to K(\bZ)$ and $TC(\bS) \to TC(\bZ)$ are the
linearization maps induced by the unit map of $E_{\infty}$ ring
spectra $\bS \to \bZ$.

For the rest of the paper, we restrict to the case of $p$ an odd
prime.  In~\cite[5.3]{BM-KSpi}, the authors showed that the fiber
square above splits as a wedge of $p-1$ fiber squares of the form:
\[
\xymatrix{%
\epsilon_{i}K(\bS)\phat \ar[r]\ar[d] & \epsilon_{i}K(\bZ)\phat \ar[d] \\ \epsilon_{i}TC(\bS)\phat \ar[r] & \epsilon_{i}TC(\bZ)\phat }
\]
for $i=0,\dotsc,p-2$ (or better, numbered modulo $(p-1)$).  We
identify the spectra $\epsilon_{i}(-)$ more specifically below.  These
squares are exactly the summands that would result from an 
eigensplitting of the fiber square for an action of $\bF_{p}^{\times}$
via the Teichm\"uller character $\omega$ for a conjectural action of
$p$-adically interpolated Adams operations; see [ibid., \S5].
We refer to the summands in the $i$th square above as the
``$\omega^{i}$ eigenspectra'' even though such Adams operations have
not been constructed in this generality.  (If they do exist, the
$\omega^{i}$ eigenspectrum is the summand where the
$\bF_{p}^{\times}$-action in the stable category is given by
the character $\omega^{i}\colon \bF_{p}^{\times}\to \bZpt$ and
the action of $\bZpt$.)

As a formal consequence, the fiber of the cyclotomic trace for $\bS$,
or equivalently, for $\bZ$ also comes with an eigensplitting.  In this
paper, we identify the eigenspectra summands.
In~\cite{BM-ktpd}, the authors identified the fiber of the cyclotomic
trace $\Fib(\tau)$ in $K$-theoretic terms as the $(-3)$-connected cover of
$\Sigma^{-1}I_{\bZp}(L_{K(1)}K(\bZ))$, where $I_{\bZp}$ denotes the
$\bZp$-Anderson dual.  Taking the idea of eigenspectra seriously, the
natural conjecture is that the $\omega^{i}$ eigenspectra of
$\Fib(\tau)$ should be the corresponding eigenspectra of this Anderson
dual.  In the spirit of the identification in [ibid.], the
$\omega^{i}$ eigenspectrum of $\Fib(\tau)$ should correspond to the
$\bZp$-Anderson dual of the $\omega^{p-i}$ eigenspectrum of
$L_{K(1)}K(\bZ)$.  Our main theorem establishes this conjecture.

\begin{thm}\label{thm:main}
There is a canonical weak equivalence between the fiber of the map
$\epsilon_{i}\tau\colon \epsilon_{i}K(\bS)\phat\to
\epsilon_{i}TC(\bS)\phat$ and the $(-3)$-connected cover of
$\Sigma^{-1}I_{\bZp}L_{K(1)}(\epsilon_{p-i}K(\bZ)\phat)$. 
\end{thm}

We also describe how the duality map of~\cite{BM-ktpd}
interacts with the fiber sequence; we discuss these additional results
in Section~\ref{sec:fiberdual}, after reviewing terminology and notation.

Theorem~\ref{thm:main} is consistent with an expansive picture of the
behavior of the conjectural $p$-adically interpolated Adams operations
on $K(\bS)$.  In particular, it is natural to conjecture compatibility
with the (known) $p$-adic Adams operations on $TC(\bS)$ as well as
multiplicative properties.  Given such operations on the
$\infty$-category level, Theorem~\ref{thm:main} would follow.
However, while the existence of such Adams operations on the stable
category level is enough to obtain a splitting on $\Fib(\tau)$, it
would not be enough to deduce Theorem~\ref{thm:main} without
additional arguments like the ones below.

While this paper obviously builds on the authors' previous work
~\cite{BM-Anil,BM-KSpi,BM-ktpd}, we have tried to make it as
self-contained as possible, with specific citations to any facts
needed from those papers.

\subsection*{Conventions}

We use the term ``stable category'' to refer to the homotopy category
of spectra with its structure as a tensor triangulated
category.  The symbol $\smaL$ denotes the smash product in the stable
category.

Some of the statements and results below involve precise
accounting for signs.  For this we use the following
conventions: suspension is $(-)\sma S^{1}$ and cone is $(-)\sma I$,
where in the latter case we use $1$ as the basepoint.  Cofiber
sequences are sequences isomorphic (in the stable category) to Puppe
sequences formed in the usual way using this suspension and cone.  A
cofiber sequence leads to a long exact sequence of homotopy groups; we
use the sign convention that for a map $f\colon A\to B$, the
connecting map $\pi_{n}Cf\to \pi_{n-1}A$ in the long exact sequence of
homotopy groups is $(-1)^{n}\sigma^{-1}$ composed with the Puppe
sequence map $\pi_{n}Cf\to \pi_{n}\Sigma A$, where $\sigma$
denotes the suspension isomorphism $\pi_{n-1}A\to \pi_{n}\Sigma A$.
(Consideration of the example of standard cells explains the
desirability of the sign.)

We form the homotopy fiber $\HFf$ of a map $f$ using the space
of paths starting from the basepoint; we then have a canonical map
$\Sigma \HFf\to Cf$ in the usual way (using the suspension coordinate to
follow the path and then follow the cone).  
We switch between fiber sequences and cofiber sequences at will, using
the convention that for the fiber sequence
\[
\Omega B\overto{\delta} \HFf\overto{\phi} A\overto{f} B,
\]
the sequence
\[
\HFf\overto{\phi} A\overto{f} B\overto{-\Sigma \delta\circ \epsilon^{-1}}\Sigma \HFf
\]
is a cofiber sequence where $\epsilon\colon \Sigma \Omega B\to B$ is
the counit of the $\Sigma$, $\Omega$ adjunction.  For the long exact
sequence of homotopy groups associated to a fiber sequence, we use the
long exact sequence of homotopy groups of the associated cofiber
sequence.  In terms of the fiber sequence displayed above, the
connecting map $\pi_{n+1}B\to \pi_{n}\HFf$ is the composite of the
canonical isomorphism $\pi_{n+1}B\iso \pi_{n}\Omega B$ and the map
$(-1)^{n}\pi_{n}\delta$.

For a cofiber sequence 
\[
A\overto{f}B\overto{g}C\overto{h} \Sigma A
\]
and a fixed spectrum $X$, the sequence
\[
\Omega F(A,X)\overto{-h^{*}}F(C,X)\overto{g^{*}}F(B,X)\overto{f^{*}}F(A,X)
\]
is a fiber sequence and
\[
F(C,X)\overto{g^{*}}F(B,X)\overto{f^{*}}F(A,X)\overto{h^{*}}\Sigma F(C,X)
\]
is a cofiber sequence.

\subsection*{Acknowledgments}
The authors thank Adebisi Agboola, Brian Conrad, Mirela Ciperiani,
Bill Dwyer, Samit Dasgupta, Mike Hopkins, Lars Hesselholt, Mahesh
Kakde, Bjorn Poonen, John Rognes, and Matthias Strauch for helpful
conversations or remarks.  We are grateful to the referee for many
small improvements and corrections.

\section{A review of the eigensplitting of the cyclotomic trace}
\label{sec:review}

In this section, we review the splitting constructed in~\cite{BM-KSpi}
of the cofiber sequence 
\begin{equation}\label{eq:fibnl}
\HFT\to K(\bZ)\phat\overto{\tau} TC(\bZ)\phat\to \Sigma \HFT.
\end{equation}
It can be useful to express this in purely $K$-theoretic terms, using
the Hesselholt-Madsen results that $TC(\bZ)\phat\to TC(\bZp)\phat$
is a weak equivalence \cite[Add.~5.2]{HM2} and that $K(\bZp)\phat\to
TC(\bZp)\phat$ is a connective cover \cite[Th.~D]{HM2}.  Thus, we can
just as well identify the connective cover of $\HFT$ as the fiber of the 
map $K(\bZ)\phat\to K(\bZp)\phat$ induced by the completion map
$\bZ\to \bZp$.  As we will recall below, little information is lost by
working in the $K(1)$-local 
category and hence by studying the cofiber sequence 
\begin{equation}\label{eq:fibloc}
\HFC\to L_{K(1)}K(\bZ)\overto{\kappa} L_{K(1)}K(\bZp)\to \Sigma \HFC.
\end{equation}
We will switch back and forth between discussing the localized
sequence~\eqref{eq:fibloc} and non-localized sequence~\eqref{eq:fibnl}.

We begin by reviewing notation for some of the basic building blocks
of the splitting.  Let $KU\phat$ denote $p$-completed complex periodic
$K$-theory, and let $L$ denote the $p$-complete Adams summand:  
\[
KU\phat \simeq L\vee \Sigma^{2}L\vee \cdots \vee \Sigma^{2p-4}L.
\]
Let $J=L_{K(1)}\bS$.  If we choose an integer $l$ which
multiplicatively generates the units of $\bZ/p^{2}$, then $J$ is
weakly equivalent to the homotopy fiber of $1-\psi^{l}\colon L\to L$
for the Adams operation $\psi^{l}$.  In deducing $p$-complete results
from $K(1)$-local results, we write $ku\phat$, $\ell$, and $j$ for the
connective covers respectively of $KU\phat$, $L$, and $J$.  

The main theorem of Dwyer-Mitchell~\cite{DwyerMitchell} (as
reinterpreted in \cite[\S2]{BM-KSpi}) produces a canonical splitting
of $L_{K(1)}K(\bZ)$ as a certain wedge of $K(1)$-local spectra
\begin{equation}\label{eq:DM-split-KZ}
L_{K(1)}K(\bZ)\simeq J\vee Y_{0}\vee \cdots \vee Y_{p-2},
\end{equation}
where $Y_{i}$ is characterized by the property that $L^{*}(Y_{i})$ is
concentrated in degrees congruent to $2i-1$ mod $2(p-1)$ with
$L^{2i-1}(Y_{i})$ defined as an $L^{0}L$-module in terms of a certain
abelian Galois group.  Because $L^{2i-1}(Y_{i})$ is a finitely
generated $L^{0}L$-module of projective dimension $1$, $Y_{i}$ is the
homotopy fiber of a map between wedges of copies of $\Sigma^{2i-1}L$
(which on $L^{2i-1}(-)$ give a projective $L^{0}L$-resolution of 
$L^{2i-1}(Y_{i})$).  From this it follows that $\pi_{*}Y_{i}$ is
concentrated in degrees congruent to $2i-1$ and $2i-2$ mod $2(p-1)$;
moreover, it is free (as a $\bZp$-module) in odd degrees.  As
$\pi_{*}J$ is concentrated in degrees congruent to $-1\equiv 2p-3$ mod
$2(p-1)$ and degree $0$, any particular homotopy group of
$L_{K(1)}K(\bZ)$ involves only at most a single $Y_{i}$ and possibly
$J$. The following two results from~\cite{BM-KSpi} (q.v.~(2.8) and the
preceding paragraph) simplify certain arguments. 

\begin{prop}\label{prop:ymo}
$Y_{0}\simeq *$.
\end{prop}

\begin{prop}\label{prop:yz}
$L^{1}Y_{1}$ is a free $L^{0}L$-module of rank 1, and so $Y_{1}$ is
(non-canonically) weakly equivalent to $\Sigma L$.
\end{prop}

We say more about the relationship of $Y_{1}$ and $\Sigma L$ in
Remark~\ref{rem:Y1} in Section~\ref{sec:Zi}.
 
Let $y_{i}$ be the $1$-connected cover of $Y_{i}$
for all $i$; then
\[
K(\bZ)\phat \simeq j\vee y_{0}\vee \cdots \vee y_{p-2}.
\]
We have a similar canonical splitting of 
$L_{K(1)}K(\bZp)$ that takes the form
\[
L_{K(1)}K(\bZp) \simeq J \vee \Sigma J' \vee Z_{0}\vee \cdots
\vee Z_{p-2}
\]
where $Z_{i}$ is non-canonically weakly equivalent to
$\Sigma^{2i-1}L$ and $J'$ is the
$K(1)$-localization of the Moore spectrum $M_{\bZpt}$ for the units of
$\bZp$, $J':=L_{K(1)}M_{\bZpt}$.  Alternatively, $J'$ is canonically
weakly equivalent to $(J\smaL M_{\bZpt})\phat$; it is non-canonically
weakly equivalent to $J$ since $(M_{\bZpt})\phat$ is non-canonically
equivalent to $\bS\phat$.  %
The spectra $Z_{i}$ were denoted $\Sigma^{-1}L_{TC}(i)$ in
\cite{BM-KSpi} for the non-canonical weak equivalence $Z_{i}\simeq
\Sigma^{-1}L(i)=\Sigma^{2i-1}L$. The $Z_{i}$ admit a canonical
description in terms of the units of cyclotomic extensions of $\bQp$,
which we review in Section~\ref{sec:Zi}.

Let $j'$ be the connective cover of $J'$, and for $i\neq 0$,
let $z_{i}$ be the $1$-connected cover of $Z_{i}$; let $z_{0}$ be the
$(-2)$-connected cover of $Z_{0}$.  Then
\[
TC(\bZ)\phat \simeq j \vee \Sigma j'\vee z_{0}\vee \cdots \vee z_{p-2}.
\]
A key result proved in
\cite[3.1]{BM-KSpi} is that the cyclotomic trace and the completion
map are diagonal on the corresponding pieces.

\begin{thm}[{\cite[3.1]{BM-KSpi}}]\label{thm:diag}
In the notation above, the cyclotomic trace $\tau\colon K(\bZ)\phat\to
TC(\bZ)\phat$ decomposes as the wedge of
the identity map $j\to j$ and maps $y_{i}\to z_{i}$ for
$i=0,\ldots,p-2$; the completion map $\kappa\colon L_{K(1)}K(\bZ)\to
L_{K(1)}K(\bZp)$ decomposes as the wedge 
of the identity map $J\to J$ and maps $Y_{i}\to Z_{i}$ for
$i=0,\ldots,p-2$.  The composite with the projection to the summands $\Sigma j'$ and
$\Sigma J'$ is the trivial map for $\tau$ and $\kappa$, respectively.
\end{thm}

It follows that the cofiber sequences of equations~\eqref{eq:fibnl}
and~\eqref{eq:fibloc} decompose into wedges of cofiber sequences.  To
explain this, we need to introduce some notation.

\begin{defn}\label{defn:X}
Let $X_{i}$ and $x_{i}$ denote the homotopy fiber of the maps $Y_{i}\to
Z_{i}$ and $y_{i}\to z_{i}$, respectively.  
\end{defn}

We have the following analogues of Propositions~\ref{prop:ymo}
and~\ref{prop:yz} for $X_{1}$ and $X_{0}$.  The first is an immediate
consequence of Proposition~\ref{prop:ymo}; the second is a restatement of
\cite[4.4]{BM-KSpi}, which asserts (in the notation here) that
$Y_{1}\to Z_{1}$ is a weak equivalence.

\begin{prop}\label{prop:X0}
The connecting map $Z_{0}\to \Sigma X_{0}$ is a weak equivalence; in
particular, $X_{0}$ is (non-canonically) weakly equivalent to
$\Sigma^{-2}L$. 
\end{prop}

\begin{prop}\label{prop:X1}
$X_{1}\simeq *$
\end{prop}

As a consequence of Definition~\ref{defn:X} we have that 
\[
\HFT\simeq j'\vee x_{0}\vee\dotsb \vee x_{p-2}\qquad \text{and}\qquad 
\HFC\simeq J'\vee X_{0}\vee\dotsb \vee X_{p-2}.
\]
Because the maps in the definition are only defined in the
stable category, the splitting of the fibers do not automatically give
a canonical splitting of the fiber sequences.  However, 
looking at the cofiber sequence
\[
\Sigma^{-1}Z_{i}\to X_{i}\to Y_{i}\to Z_{i}
\]
we see that $X_{i}$ can have $L$-cohomology only in degrees congruent
to $2i-1$ and $2i-2$ mod $2(p-1)$, or equivalently:

\begin{prop}\label{prop:mapsfromX}
For $i=0,\dotsc,p-2$, $[X_{i},\Sigma^{j}L]=0$ unless $j\equiv 2i-1$ or $j\equiv 2i-2$ mod
$2(p-1)$.
\end{prop}

From this we see that $[X_{i},\Sigma^{-1}Z_{j}]=0$ for $j\neq i$ and
$[x_{i},\Sigma^{-1}z_{j}]=0$ for $j\neq i$.  The latter is clear from
Proposition~\ref{prop:X0} in the case $i=0$ and in the remaining cases
follows from the isomorphisms
\[
[x_{i},\Sigma^{-1}z_{j}]\iso
[x_{i},\Sigma^{-1}Z_{j}]\iso [X_{i},\Sigma^{-1}Z_{j}].
\]
The first isomorphism holds since $x_{i}$ is $0$-connected for
$i=1,\dotsc,p-2$ while 
$\Sigma^{-1}z_{j}\to \Sigma^{-1}Z_{j}$ is a weak equivalence on
$0$-connected covers for all $j$.  The
second isomorphism holds because $X_{i}$ is the $K(1)$-localization of
$x_{i}$ and $\Sigma^{-1}Z_{j}$ is $K(1)$-local. We also note that
$[J',\Sigma^{-1}Z_{j}]=0$ and $[j',\Sigma^{-1}Z_{j}]=0$ for all $j\neq
1$ since $\pi_{0}\Sigma^{-1}Z_{j}=0$.  We have that
$[J',\Sigma^{-1}Z_{1}]\iso \bZp$ (non-canonically) and so
$[j',\Sigma^{-1}Z_{1}[0,\infty)]\iso \bZp$ with the isomorphism
induced by $\pi_{0}$:
\[
[j',\Sigma^{-1}Z_{1}[0,\infty)]\to
\Hom(\pi_{0}j',\pi_{0}\Sigma^{-1}Z_{1})=\Hom((\bZpt)\phat,\bZp)\iso \bZp.
\]
Since $\Sigma^{-1} z_{1}$ is the fiber of the map $\Sigma^{-1}Z_{1}[0,\infty)\to
H\bZp$ that induces the identification $\pi_{1}Z_{1}=\bZp$, we see
that $[j',\Sigma^{-1}z_{1}]=0$.  In particular, we have shown that all of the
indeterminacy in the map of fiber sequences is zero (in the
non-localized case), and we get the following consequence.

\begin{cor}\label{cor:split}
The cofiber sequence~\eqref{eq:fibnl} splits canonically as a wedge of 
the cofiber sequences
\[
\xymatrix@C-.5pc@R-1.5pc{%
\relax*\ar[r]&j\ar[r]^{=}&j\ar[r]&\relax *\\
j'\ar[r]&\relax *\ar[r]&\Sigma j'\ar[r]^{=}&\Sigma j'\\
x_{0}\ar[r]&y_{0}\ar[r]&z_{0}\ar[r]&\Sigma x_{0}\\
\vdots&\vdots&\vdots&\vdots\\
x_{p-2}\ar[r]&y_{p-2}\ar[r]&z_{p-2}\ar[r]&\Sigma x_{p-2}.
}
\]
The cofiber sequence~\eqref{eq:fibloc} splits canonically as an
analogous wedge in terms of $J$, $J'$, $X_i$, $Y_i$, and $Z_i$.
\end{cor}

\section{Review of arithmetic duality in algebraic \texorpdfstring{$K$}{K}-theory}\label{sec:local-review}

In~\cite{BM-ktpd}, we identified the fiber of the cyclotomic trace
using a spectral lift of global arithmetic duality.  Relating this
identification to the splittings described above requires some of the
details of the global duality map and a closely related local duality
map.  We give a short and mostly self-contained review of the
construction in this section.

In arithmetic, local duality is an isomorphism
\[
H^{i}_{\et}(k;M)\overto{\iso} (H^{2-i}_{\et}(k;M^{*}(1)))^{*}
\]
where $k$ is the field of fractions of a complete discrete valuation
ring whose residue field is finite (e.g., a finite extension of
$\bQp$), $M$ is a finite Galois module, and $(-)^{*}$ denotes the
Pontryagin dual.  A version of this duality holds in algebraic
$K$-theory where it takes the following form.  (For a proof,
see~\cite[1.4]{BM-ktpd}.) 

\begin{thm}[$K$-Theoretic Local Duality]
\label{thm:localduality}
Let $k$ be the field
of fractions of a complete discrete valuation ring whose residue field
is finite.  The map 
\[
L_{K(1)}K(k) \to I_{\bQ/\bZ}(L_{K(1)}K(k) \smaL M_{\bZpi})\simeq I_{\bZp}(L_{K(1)}K(k))
\]
adjoint to the composite map
\[
L_{K(1)}K(k)\smaL L_{K(1)}K(k)\smaL M_{\bZpi}\to I_{\bQ/\bZ}\bS
\]
described below is a weak equivalence.
\end{thm}

We have stated the theorem in a way that emphasizes the parallel with
the algebraic result.
The functor $I_{\bQ/\bZ}(-)$ denotes Brown-Comenetz duality, the spectral
analogue of Pontryagin duality; it can be constructed in terms of the
Brown-Comenetz dual of the sphere spectrum $I_{\bQ/\bZ}\bS$ as the derived
function spectrum $F(-,I_{\bQ/\bZ}\bS)$.  The functor $I_{\bZp}(-)$
denotes $\bZp$-Anderson duality; it can be constructed as 
\[
I_{\bZp}(-)=I_{\bQ/\bZ}(-\smaL M_{\bZpi})\iso
F(-,F(M_{\bZpi},I_{\bQ/\bZ}\bS))\iso F(-,I_{\bZp}\bS).
\]

The duality map in Theorem~\ref{thm:localduality} is constructed as
follows.  The $E_{\infty}$
multiplication on $K(k)$ induces a map
\[
L_{K(1)}K(k) \smaL L_{K(1)}K(k) \smaL M_{\bZpi}\to L_{K(1)}K(k)\smaL M_{\bZpi}
\]
and we have a canonical map 
\begin{equation}\label{eq:canmap}
L_{K(1)}K(k)\smaL M_{\bZpi}\to I_{\bQ/\bZ}\bS
\end{equation}
essentially induced by the Hasse invariant (see \cite[(1.2)]{BM-ktpd}
for details).  

In terms of Anderson duality, the weak equivalence 
\[
L_{K(1)}K(k)\overto{\simeq}I_{\bZp}(L_{K(1)}K(k))
\]
in Theorem~\ref{thm:localduality} is adjoint to the map
\[
L_{K(1)}K(k) \smaL L_{K(1)}K(k) \to I_{\bZp}\bS
\]
induced by the multiplication and the map
\begin{equation}\label{eq:v}
v_{k}\colon L_{K(1)}K(k)\to I_{\bZp}\bS.
\end{equation}
adjoint to~\eqref{eq:canmap}.

The local duality theorem relates to our work in this paper when we
consider the case $k=\bQp$.  In this case, Quillen's localization
theorem and Quillen's calculation of the $K$-theory of $\bF_{p}$
together say that the fiber of the map $K(\bZp)\phat\to K(\bQp)\phat$
is $H\bZp$ and it follows that the map $L_{K(1)}K(\bZp)\to
L_{K(1)}L(\bQp)$ is a weak equivalence.  We then have the following
corollary. 

\begin{cor}[Local duality for $\bZp$]\label{cor:localduality}
Let $v_{\bZp}\colon L_{K(1)}K(\bZp)
\to I_{\bZp}\bS$ be the composite of the map
$L_{K(1)}K(\bZp)\overto{\simeq}L_{K(1)}K(\bQp)$ and the map
$v_{\bQp}\colon L_{K(1)}K(\bQp)\to I_{\bZp}\bS$ of~\eqref{eq:v}.
The map $L_{K(1)}K(\bZp)\to I_{\bZp}(L_{K(1)}K(\bZp))$ adjoint to the
composite of multiplication and $v_{\bZp}$
\[
L_{K(1)}K(\bZp)\smaL L_{K(1)}K(\bZp)\to L_{K(1)}K(\bZp)
\overto{v_{\bZp}} I_{\bZp}\bS
\]
is a weak equivalence.
\end{cor}

The $K$-theoretic analogue of global duality identifies $\HFC$
in~\eqref{eq:fibloc} in terms of the $\bZp$-Anderson dual of
$L_{K(1)}K(\bZ)$.  Rather than stating it in the full generality
proved in~\cite{BM-ktpd}, we state it just in this case.
In~\cite[(1.7)]{BM-ktpd}, we construct a map 
\begin{equation}\label{eq:u}
u_{\bQ}\colon \HFC\to \Sigma^{-1} I_{\bZp}\bS.
\end{equation}
from the Albert-Brauer-Hasse-Noether sequence for $\bQ$. It is
compatible with the map $v_{\bQp}$ above in the
sense that $v_{\bQp}$ is the composite  
\[
L_{K(1)}K(\bQp)\simeq L_{K(1)}K(\bZp)\to \Sigma \HFC\overto{\Sigma u_{\bQ}}
\Sigma \Sigma^{-1}I_{\bZp}\bS\iso I_{\bZp}\bS
\]
where $L_{K(1)}K(\bZp)\to \Sigma \HFC$ is the connecting map in the
fiber sequence~\eqref{eq:fibloc}.  The map~\eqref{eq:u} induces the
following $K$-theoretic global duality theorem.  

\begin{thm}[$K$-Theoretic Tate-Poitou Duality for $\bZ$]
\label{thm:globalduality}
The map 
\[
\HFC\to \Sigma^{-1}I_{\bZp} L_{K(1)}K(\bZ)
\]
adjoint to the map 
\[
L_{K(1)}K(\bZ)\smaL \HFC\to \Sigma^{-1} I_{\bZp}\bS
\]
induced by the $L_{K(1)}K(\bZ)$-module structure map $L_{K(1)}K(\bZ)\smaL \HFC\to \HFC$
and the map $u_{\bQ}\colon \HFC\to \Sigma^{-1}I_{\bZp}\bS$ is a weak equivalence.
\end{thm}

\section{The eigensplitting of the fiber of the cyclotomic trace}

Combining Theorem~\ref{thm:globalduality} with the canonical wedge
decomposition of $K(\bZ)$ described above in
equation~\eqref{eq:DM-split-KZ}, we obtain the following decomposition 
of $\HFC$.
\begin{align*}
\HFC \simeq \Sigma^{-1}I_{\bZp}(L_{K(1)}K(\bZ)) 
&\simeq
\Sigma^{-1}I_{\bZp}(J\vee Y_{0}\vee \cdots \vee Y_{p-2})\\
&\simeq
\Sigma^{-1}I_{\bZp}J \vee \Sigma^{-1}I_{\bZp}Y_{0} \vee \cdots  \vee \Sigma^{-1}I_{\bZp}Y_{p-2}.
\end{align*}
Our goal in this section is to identify this wedge decomposition with
(a permutation of)
the wedge decomposition 
\[
\HFC\simeq J'\vee X_{0}\vee \cdots \vee X_{p-2}
\]
constructed above in Corollary~\ref{cor:split}. This is accomplished
in the following theorem together with an observation on the $J'$
summand stated in Theorem~\ref{thm:stdmap} below. 

\begin{thm}
The canonical isomorphism in the stable category
\[
\HFC\simeq \Sigma^{-1}I_{\bZp}J \vee \Sigma^{-1}I_{\bZp}Y_{0} \vee \cdots  \vee \Sigma^{-1}I_{\bZp}Y_{p-2}.
\]
identifies
$J'$ as $\Sigma^{-1}I_{\bZp}J$ and 
$X_{i}$ as $\Sigma^{-1}I_{\bZp}Y_{p-i}$ (for $i\neq 0,1$) or
$\Sigma^{-1}I_{\bZp}Y_{1-i}$ (for $i=0,1$).  Moreover:
\begin{enumerate}
\item $[X_{i},\Sigma^{-1}I_{\bZp}Y_{j}]=0$ unless $i+j\equiv 1$ mod
$(p-1)$. 
\item $[X_{i},\Sigma^{-1}I_{\bZp}J]=0$ for all $i$.
\item $[J',\Sigma^{-1}I_{\bZp}Y_{j}]=0$ for all $j$.
\end{enumerate}
\end{thm}

In later formulas we will simply write $X_{i}\simeq
\Sigma^{-1}I_{\bZp}Y_{p-i}$ and understand the indexing to be mod
$(p-1)$. 

\begin{proof}
To simplify notation, we write $D$ for $\Sigma^{-1}I_{\bZp}$ inside
this proof.  The multiplication $L\smaL L\to L$ together with the
canonical identification of $\pi_{0}L$ as $\bZp$ induces a map $L\to
I_{\bZp}L$ that is easily seen to be an isomorphism in the stable
category, q.v.~\cite[2.6]{Knapp-Anderson} (this is essentially due to
Anderson~\cite{Anderson-UCT}). This gives us a canonical
identification of $DL$ as $\Sigma^{-1}L$, which is the main tool we
use. 

All statements follow from verification of~(i), (ii), and (iii).  
As discussed above, $Y_{j}$ fits in a cofiber sequence of the form
\[
\bigvee \Sigma^{2j-2}L\to Y_{j}\to \bigvee \Sigma^{2j-1}L\to \bigvee \Sigma^{2j-1}L
\]
(for some finite wedges of copies of $\Sigma^{n}L$); it follows that
$DY_{j}$ fits into a cofiber sequence of the form
\[
\bigvee \Sigma^{-2j}L\to \bigvee \Sigma^{-2j}L\to DY_{j}\to
\bigvee \Sigma^{-2j+1}L.
\]
Applying Proposition~\ref{prop:mapsfromX}, we see that
$[X_{i},DY_{j}]=0$ unless $-2j\equiv 2i-2$ mod $2(p-1)$, or
equivalently $i+j\equiv 1$ mod $(p-1)$. This proves~(i).  

Writing $J$ as the fiber of a self-map of $L$,  $DJ$ fits then into a
cofiber sequence of the form 
\[
\Sigma^{-1}L\to DJ\to L\to L
\]
and again applying Proposition~\ref{prop:mapsfromX}, it follows that
$[X_{i},DJ]=0$ unless $2i-2\equiv 0$ or $2i-1\equiv -1$ mod $2(p-1)$.
In the first case, $X_{1}\simeq *$ by Proposition~\ref{prop:X1}.
In the second case $i=0$; by Proposition~\ref{prop:X0}, $X_{0}$ is
non-canonically weakly equivalent to $\Sigma^{-2}L$, and
$[\Sigma^{-2}L,DJ]=0$.  This proves~(ii). 

Finally, to prove~(iii), we note that $DY_{j}$ is $K(1)$-local.  Since
$J'$ is non-canonically weakly equivalent to $J\simeq L_{K(1)}\bS$, to
see that $[J',DY_{j}]=0$ it suffices to see that $\pi_{0}DY_{j}=0$.
Since $\pi_{*}Y_{j}$ is concentrated in degrees congruent to $2j-1$ and
$2j-2$ mod $2(p-1)$, we have that $\pi_{0}DY_{j}$ can only possibly be
non-zero for $j=0$ but by Proposition~\ref{prop:ymo}, $Y_{0}\simeq *$.
\end{proof}

While we have defined the $X_{i}$ solely in terms of the fiber
sequence, we have defined $J'$ intrinsically, and so the equivalence
of $J'$ with $\Sigma^{-1}I_{\bZp} J$ under the isomorphism in the
stable category $\HFC\simeq \Sigma^{-1}I_{\bZp}K(\bZ)$ constitutes
additional information.  In fact, we have a canonical weak equivalence 
\[
J'\to \Sigma^{-1}I_{\bZp}J
\]
that we call the \term{the standard weak equivalence}, constructed as
follows.  Since $J'$ is defined as the
$K(1)$-localization of the Moore spectrum $M_{\bZpt}$, and
$(\bZpt)\phat$ is a projective $\bZp$-module, maps in the
stable category from $J'$ into $K(1)$-local spectra are in canonical
one-to-one correspondence with homomorphisms from $(\bZpt)\phat$ into
$\pi_{0}$.  We note that $\Sigma^{-1}I_{\bZp}J$ is $K(1)$-local,
and to calculate $\pi_{0}\Sigma^{-1}I_{\bZp}J$, we use the fundamental
short exact sequence for the $\bZp$-Anderson dual:
For any spectrum $X$, there is a canonical natural short exact
sequence 
\begin{equation}\label{eq:fses}
0\to \Ext(\pi_{-n-1}X,\bZp)\to \pi_{n}I_{\bZp} X\to
\Hom(\pi_{-n}X,\bZp)\to 0.
\end{equation}
For finitely generated $\bZp$-modules, $\Hom(-,\bZp)$ and
$\Ext(-;\bZp)$ coincide with\break $\Hom_{\bZp}(-,\bZp)$ and
$\Ext_{\bZp}(-;\bZp)$.  In the case of $X=J$, since $\pi_{-2}J=0$, we
then have a canonical identification of $\pi_{0}\Sigma^{-1}I_{\bZp}J$
as $\Hom(\pi_{-1}J,\bZp)$.  The Morava Change of Rings Theorem
identifies $\pi_{-1}J$ canonically in terms of continuous group
cohomology:
\[
\pi_{-1}J\iso H^{1}_{c}(\bZpt;\bZp)\iso \Hom(\bZpt,\bZp) \iso \Hom((\bZpt)\phat,\bZp),
\]
q.v.~\cite[(1.1)]{DevinatzHopkins-Homotopy}, for the continuous action
of $\bZpt$ on $\pi_{*}KU\phat$ arising from the $p$-adic interpolation
of the Adams operations.  This then gives a
canonical isomorphism
\[
\pi_{0}\Sigma^{-1}I_{\bZp}J\iso \Hom(\Hom((\bZpt)\phat,\bZp),\bZp).
\]
Since $(\bZpt)\phat$ is
projective of rank 1, the double dual map
is an isomorphism, giving us a canonical isomorphism 
\[
(\bZpt)\phat\to \pi_{0}\Sigma^{-1}I_{\bZp}J
\]
specifying the standard weak equivalence.

On the other hand, we have a canonical map $J'\to \HFC$ arising from
the fiber sequence~\eqref{eq:fibloc} and the $\Sigma J'$ summand of
$L_{K(1)}K(\bZp)\simeq L_{K(1)}K(\bQp)$.  In terms of maps from
$(\bZpt)\phat$ into $\pi_{0}\HFC$, we can therefore identify
this map $J'\to \HFC$ as coming from the canonical identification of
the cokernel of
\[
\pi_{1}L_{K(1)}K(\bZ)\to \pi_{1}L_{K(1)}K(\bZp)
\]
as $(\bZpt)\phat$ (the $p$-completion of the cokernel of the map
$(\bZ[1/p])^{\times}\to \bQpt$).  The following theorem compares the
two maps.

\begin{thm}\label{thm:stdmap}
The composite map $J'\to \HFC\simeq \Sigma^{-1}I_{\bZp}K(\bZ)\to
\Sigma^{-1}I_{\bZp}J$ is the standard weak equivalence.
\end{thm}

We postpone the proof to Section~\ref{sec:classfield}.

\section{Self-duality of the fiber sequence of the cyclotomic trace}
\label{sec:sldnl}\label{sec:fiberdual}

In this section, we extend the analysis from the previous section by
showing that the fiber sequence defining $\HFC$ is self-dual.  This
requires the compatibility of $K$-theoretic local and global duality
discussed in Section~\ref{sec:local-review}.

\begin{thm}\label{thm:selfdual}
The following diagram commutes up to the indicated sign
\[
\xymatrix@C-1pc{%
\HFC\ar[r]^{\rho}\ar[d]_{\simeq}
&L_{K(1)}K(\bZ)\ar[r]^{\kappa}\ar[d]_{\simeq}
&L_{K(1)}K(\bZp)\ar[r]^{\partial}\ar[d]_{\simeq}\ar@{{}{}{}}[dr]|{\textstyle (-1)}
& \Sigma \HFC \ar[d]_{\simeq} \\
\Sigma^{-1}I_{\bZp}(L_{K(1)}K(\bZ)) \ar[r]_-{\Sigma^{-1}I_{\bZp}\rho\mystrut}
&\Sigma^{-1}I_{\bZp}(\HFC)\ar[r]_{I_{\bZp}\partial\mystrut}
&I_{\bZp} L_{K(1)} K(\bZp) \ar[r]_{-I_{\bZp}\kappa\mystrut}
& I_{\bZp}L_{K(1)}K(\bZ)  
}
\]
where the top sequence is the cofiber sequence (associated to the
fiber sequence) defining $\HFC$, the bottom sequence is the
$\bZp$-Anderson dual of its rotation, and the vertical maps are induced by
the $K$-theoretic Tate-Poitou duality theorem for $\bZ$
(Theorem~\ref{thm:globalduality}) and the $K$-theoretic
local duality theorem for $\bZp$ (Theorem~\ref{cor:localduality}). 
\end{thm}

\begin{proof}
The assertion is that $\Sigma \rho$ is $\bZp$-Anderson dual to
$\rho$ and $\kappa$ is $\bZp$-Anderson dual to $\partial$.
Given pairings 
\[
\epsilon_{i} \colon A_{i}\smaL B_{i}\to I_{\bZp}\bS
\]
whose adjoints $\eta_{i} \colon B_{i}\to F(A_{i},I_{\bZp}\bS)$
are weak
equivalences, then for maps $f\colon A_{1}\to A_{2}$ and $g\colon
B_{2}\to B_{1}$, $\eta_{1}\circ g\circ \eta_{2}^{-1}$ is $\bZp$-Anderson dual
to $f$ exactly when the diagram
\[
\xymatrix{%
A_{1}\smaL B_{2}\ar[r]^{\id\smaL g}\ar[d]_{f\smaL \id}
&A_{1}\smaL B_{1}\ar[d]^{\epsilon_{1}}\\
A_{2}\smaL B_{2}\ar[r]_{\epsilon_{2}}
&I_{\bZp}\bS
}
\]
commutes.  In this case, when the weak equivalences $\eta_{1}$,
$\eta_{2}$ are fixed and understood, we say that $g$ is
$\bZp$-Anderson dual to $f$. 
By construction, the following diagram commutes
\[
\xymatrix{%
L_{K(1)}K(\bZ)\smaL L_{K(1)}K(\bZp)
  \ar[r]^-{\id\smaL \partial}\ar[d]_{\kappa \smaL \id}
&L_{K(1)}K(\bZ)\smaL \Sigma \HFC\ar[r]^-{\xi}
&\Sigma \HFC\ar[d]^{\Sigma u_{\bZ}}\\
L_{K(1)}K(\bZp)\smaL L_{K(1)}K(\bZp)\ar[r]_-{\mu}
&L_{K(1)}K(\bZp)\ar[r]_-{v_{\bZp}}\ar[ur]^-{\partial}
&I_{\bZp}\bS
}
\]
where $\mu$ denotes the multiplication, $\xi$ denotes the
$L_{K(1)}K(\bZ)$-module action map, and $u$ and $v$ are the maps
in the global and local duality theorems, respectively.  
This gives the duality between $\partial$ and $\kappa$.
To compare $\rho$ and $\Sigma \rho$, consider the diagram
\[
\xymatrix{%
\HFC\smaL \Sigma \HFC\ar[r]^-{\id\smaL \Sigma \rho}\ar[d]_{\rho \smaL \id}
&\HFC\smaL \Sigma L_{K(1)}K(\bZ)\ar[d]\\
L_{K(1)}K(\bZ)\smaL \Sigma \HFC\ar[r]&I_{\bZp}\bS
}
\]
where the unlabeled maps are induced by the duality pairing.  The
down-then-right composite is $u_{\bZ}$ composed with the suspension of
the non-unital multiplication on $\HFC$, whereas the right-then-down
composite is $u_{\bZ}$ composed with the suspension of
the opposite of the non-unital multiplication on $\HFC$.  Since the
non-unital multiplication on $\HFC$ is $E_{\infty}$ and in particular
commutative in the stable category, the diagram commutes.
\end{proof}

As an immediate consequence, we get duality between the 
cofiber sequences in Corollary~\ref{cor:split}.
Theorem~\ref{thm:stdmap} indicates the relationship between the $j$
and $j'$ sequences.  The relationship between the remaining ones is
summarized in the following corollary.

\begin{cor}
For each $i$ in $\bZ/(p-1)$, the following diagram commutes up to the
indicated sign
\[
\xymatrix{%
X_{i}\ar[r]^-{\rho_{i}}\ar[d]_{\simeq}
&Y_{i}\ar[r]^-{\kappa}_{i}\ar[d]_{\simeq}
&Z_{i}\ar[r]^-{\partial_{i}}\ar[d]_{\simeq}\ar@{{}{}{}}[dr]|{\textstyle (-1)}
&\Sigma X_{i}\ar[d]_{\simeq}\\
\Sigma^{-1}I_{\bZp}(Y_{p-i})\ar[r]_-{\mystrut\Sigma^{-1}I_{\bZp}\rho_{p-i}}
&\Sigma^{-1}I_{\bZp}(X_{p-i})\ar[r]_-{\mystrut I_{\bZp}\partial_{p-i}}
&I_{\bZp}(Z_{p-i})\ar[r]_{\mystrut -I_{\bZp}\kappa_{p-i}}
&I_{\bZp}(Y_{p-i})
}
\]
where the top sequence is the cofiber sequence (associated to the
fiber sequence) defining $X_{i}$, the bottom sequence is the
$\bZp$-Anderson dual of its rotation, and the vertical maps are
induced by the maps
\begin{gather*}
X_{i}\to \HFC\to \Sigma^{-1}I_{\bZp}(L_{K(1)}K(\bZ)) \to
  \Sigma^{-1}I_{\bZp}(Y_{p-i})\\
Y_{i}\to L_{K(1)}K(\bZ)\to \Sigma^{-1}I_{\bZp}(\HFC)\to
   \Sigma^{-1}I_{\bZp}(X_{p-i})\\
Z_{i}\to L_{K(1)}K(\bZp)\to I_{\bZp}(L_{K(1)}K(\bZp))\to
   I_{\bZp}(Z_{p-i})
\end{gather*}
arising from local and global $K$-theoretic duality.
\end{cor}

In the case of primes that satisfy the Kummer-Vandiver condition, we
can be a bit more specific.  A prime $p$ \term{satisfies the
Kummer-Vandiver condition} when $p$ does not divide the order of the
class group of $\bZ[\zeta_{p}+\zeta_{p}^{-1}]$ for $\zeta_{p}=e^{2\pi
i/p}$.  In this case, Dwyer-Mitchell~\cite[12.2]{DwyerMitchell}
identifies the homotopy type of the spectra $Y_{i}$ in terms of the
Kubota-Leopoldt $p$-adic $L$-function: Given any power series $f$ in
the $p$-adic integers, there is a unique self-map $\phi_{f}$ on $L$ in
the stable category such that on $\pi_{2(p-1)n}$, $\phi_{f}$ is
multiplication by $f((1+p)^{n}-1)$ (cf.~\cite[2.4]{Mitchell-padic}).
A celebrated theorem of Iwasawa~\cite{Iwasawa-Lfunctions} implies in
this context that for $i=2,4,\ldots,p-3$, there exists a self-map of
$\Sigma^{2i-1}L$ which on $\pi_{2n+1}$ is multiplication by the value
of the Kubota-Leopoldt $p$-adic $L$-function $L_{p}(-n,\omega^{i})$.
The spectrum $Y_{i}$ is non-canonically weakly equivalent to the
homotopy fiber of this map.  The $\bZp$-Anderson self-duality of $L$
then identifies $I_{\bZp}Y_{i}$ as (non-canonically) weakly equivalent
to the fiber of the self-map of $\Sigma^{2-2i}L$ that on $\pi_{2n}$ is
multiplication by $L_{p}(n,\omega^{i})$.  In particular, for
$j=3,5,\ldots, p-2$, $X_{j}\simeq \Sigma^{-1}I_{\bZp}Y_{p-j}$ is then non-canonically weakly equivalent
to the homotopy fiber of the self-map of $\Sigma^{2j-1}L$ that on
$\pi_{2n-1}$ is multiplication by $L_{p}(n,\omega^{1-j})$, or
equivalently, on $\pi_{2n+1}$ is multiplication by
$L_{p}(n+1,\omega^{1-j})$.  For $i$ 
odd and for $j$ even, $Y_{i}$ and $X_{j}$ are non-canonically weakly
equivalent to $\Sigma^{2i-1} L$ and $\Sigma^{2j}L$, respectively.
(Independently of the Kummer-Vandiver condition $Y_{0}\simeq *$ and
$X_{1}\simeq *$ by Propositions~\ref{prop:ymo} and~\ref{prop:X1}.)

In the case of an odd regular prime, the relevant values of the
Kubota-Leopoldt $p$-adic $L$-functions are units, and the spectra
$X_{j}$ and $x_{j}$ are trivial for $j$ odd.  This is consistent with Rognes'
computation~\cite[3.3]{Rognesp} of the homotopy fiber of the
cyclotomic trace as (non-canonically) weakly equivalent to $j\vee
\Sigma^{-2}ko\phat$ in this case.   More generally, we have the
following corollary.

\begin{cor}
Let $p$ be an odd prime that satisfies the Kummer-Vandiver condition.
The cofiber sequence 
\[
\HFT \to
K(\bZ)\phat\to TC(\bZ)\phat\to
\Sigma \HFT
\]
is (non-canonically) weakly equivalent to the wedge of the cofiber sequences
\[
\xymatrix@C-.5pc@R-1.5pc{%
\relax *\ar[r]&j\ar[r]^{=}&j\ar[r]&\relax* \\
j'\ar[r]&\relax*\ar[r]&\Sigma j'\ar[r]^{=}&\Sigma j' \\
x_{i}\ar[r]&\Sigma^{2i-1}\ell\ar[r]^{\lambda^{o}_{i}}&\Sigma^{2i-1}\ell\ar[r]&\Sigma x_{i} &i=3,5,\ldots,p-2\\
\Sigma^{2i}\ell\ar[r]&y_{i}\ar[r]&\Sigma^{2i-1}\ell\ar[r]^{\Sigma \lambda^{e}_{i}}&\Sigma^{2i-1}\ell&i=2,4,\ldots,p-3\\
\relax *\ar[r]&\Sigma^{2p-1} \ell\ar[r]^{=}&\Sigma^{2p-1}\ell\ar[r]&\relax* &(i=1)\\
\Sigma^{-2}\ell\ar[r]&\relax*\ar[r]&\Sigma^{-1} \ell\ar[r]^{=}&\Sigma^{-1}\ell&(i=0)
}
\]
where $\lambda^{o}_{i}$ is the unique self-map of
$\Sigma^{2i-1}\ell$ that on $\pi_{2n+1}$ is multiplication by the value
$L_{p}(n+1,\omega^{p-i})$, $\lambda^{e}_{i}$ is the unique self-map of
$\Sigma^{2i-1}\ell$ that on $\pi_{2n+1}$ is multiplication by the value
$L_{p}(-n,\omega^{i})$, and $L_{p}$ denotes the Kubota-Leopoldt
$p$-adic $L$-function. 
\end{cor}

\section {Proof of Theorem~\ref{thm:stdmap}}\label{sec:classfield}
%
As discussed above the statement of Theorem~\ref{thm:stdmap}, maps
from $J'$ to $\HFC$ are 
determined by maps from $(\bZpt)\phat$ to $\pi_{0}\HFC$; we have two
isomorphisms of $(\bZpt)\phat$ with $\pi_{0}\HFC$ and we need to show
that they are the same.  It is slightly easier to work with
the $\bZp$-duals instead. We can canonically identify the
$\bZp$-dual of $\pi_{0}\HFC$ as
$\pi_{-1}L_{K(1)}K(\bZ)$: the fundamental short exact
sequence~\eqref{eq:fses} for $\pi_{0}\Sigma^{-1}I_{\bZp}K(\bZ)$ gives
an isomorphism
\[
\pi_{0}\Sigma^{-1}I_{\bZp}K(\bZ)\iso \Hom(\pi_{-1}L_{K(1)}K(\bZ),\bZp)
\]
since $\pi_{-2}L_{K(1)}K(\bZ)=0$ (as $Y_{0}\simeq *$).
The identification of $\pi_{-1}J$
as $\Hom(\bZpt,\bZp)$ above and the canonical map $J\to
L_{K(1)}K(\bZ)$ gives one isomorphism of 
$\pi_{-1}L_{K(1)}(\bZ)$ with $\Hom(\bZpt,\bZp)$, which we denote as
$\alpha$.  The isomorphism of $\pi_{0}\HFT$ with $(\bZpt)\phat$ as the quotient of
$\pi_{1}L_{K(1)}K(\bQp)\iso (\bQpt)\phat$ gives another isomorphism of
$\pi_{-1}L_{K(1)}(\bZ)$ with $\Hom(\bZpt,\bZp)$, which we denote as
$\eta$.  We need to prove that $\eta =\alpha$.

We have an intrinsic identification of $\pi_{-1}L_{K(1)}K(\bZ)$
coming from Thomason's descent spectral
sequence~\cite[4.1]{ThomasonEtale}, which in this case canonically
identifies  
\[
\pi_{-1}L_{K(1)}K(\bZ)\iso \pi_{-1}L_{K(1)}K(\bZ[1/p])
\]
as $H^{1}_{\et}(\Spec \bZ[1/p];\bZp)$
(continuous \'etale cohomology).  Let $\bQS$ denote the maximal
algebraic extension of $\bQ$ that is unramified except over $S=\{p\}$, and let
$G_{S}=\Gal(\bQS/\bQ)$.  The abelianization $G_{S}^{\ab}$ of $G_{S}$
corresponds to the maximal abelian extension of $\bQ$ that is
unramified except over $p$, which is $\bQ(\mu_{p^{\infty}})$ (where
$\mu_{p^{\infty}}$ denotes the group of $p^{n}$th roots unity for all
$n$).  We have the standard identification of the 
continuous \'etale cohomology $H^{1}_{\et}(\Spec \bZ[1/p];\bZp)$ as the
Galois cohomology
$H^{1}_{\Gal}(\bQS/\bQ;\bZp)$~\cite[II.2.9]{Milne-Duality2007}, which
we can identify as the abelian group of continuous homomorphisms 
\[
\Hom_{c}(G_{S},\bZp)\iso \Hom_{c}(G_{S}^{\ab},\bZp).
\]
We have a further isomorphism 
\[
G_{S}^{\ab}=\Gal(\bQ(\mu_{p^{\infty}})/\bQ)\iso
\lim_{n}\Gal(\bQ(\mu_{p^{n}})/\bQ) \iso \lim_{n} (\bZ/p^{n})^{\times} \iso\bZpt,
\]
where the first isomorphism is inverse to the isomorphism
$(\bZ/p^{n})^{\times}\iso \Gal(\bQ(\mu_{p^{n}})/\bQ)$ taking an
element $u\in (\bZ/p^{n})^{\times}$ to the automorphism of
$\bQ(\mu_{p^{n}})$ induced by the automorphism $\zeta \mapsto
\zeta^{u}$ on $\mu_{p^{n}}$.  This then constructs an isomorphism
\[
\gamma \colon \pi_{-1}L_{K(1)}K(\bZ)\to \Hom_{c}(\bZpt,\bZp)\iso \Hom(\bZpt,\bZp).
\]

First we show $\gamma =\alpha$.  Choose a prime number $l$ such that
$l$ is a topological generator of $\bZpt$, or equivalently, a
generator of $\bZ/p^{2}$, and consider the quotient map $\bZ[1/p]\to
\bZ/l=\bFl$. By celebrated work of Quillen~\cite{Quillen-KFinite}, the composite map 
\[
j\to K(\bZ)\phat\to K(\bFl)\phat
\]
is a weak equivalence and an embedding of $\bbFlt$ in $\bC^{\times}$
induces a weak equivalence $K(\bbFl)\phat\to ku\phat$ with the
automorphism $\Phi$ on $K(\bbFl)\phat$ induced by the Frobenius
automorphism of $\bbFl$ mapping to the Adams operation $\Psi^{l}$ on
$ku\phat$ (independently of the choice of embedding). We will also
write $\Phi$ for the corresponding automorphism of
$L_{K(1)}K(\bbFl)$. For any functorial model of $L_{K(1)}K(-)$, the
induced map from $L_{K(1)}K(\bFl)$ into the homotopy fixed points of
$\Phi$ (the homotopy equalizer of $\Phi$ and the identity on
$L_{K(1)}K(\bbFl)$) is a weak equivalence.  Writing
$L_{K(1)}K(\bbFl)^{h\Phi}$ for the homotopy fixed points of $\Phi$,
the map 
\[
L_{K(1)}K(\bFl)\to L_{K(1)}K(\bbFl)^{h\Phi}
\]
is the unique map that takes the unit element of
$\pi_{0}(L_{K(1)}K(\bFl))$ to the unique element of
$\pi_{0}(L_{K(1)}K(\bbFl)^{h\Phi})$ that maps to the unit element of
$\pi_{0}(L_{K(1)}K(\bbFl))$.  This gives a canonical identification of
$\pi_{-1}(L_{K(1)}K(\bFl))$ as $H^{1}(\langle \Phi
\rangle;\bZp)$, where $\langle \Phi \rangle$ denotes the cyclic group
generated by $\Phi$; we have used the canonical isomorphism
$\pi_{0}(L_{K(1)}K(\bbFl))\iso\bZp$ induced by the unit and we note
that this isomorphism is consistent with the canonical isomorphism
$\pi_{0}KU\phat\iso \bZp$ under the weak equivalence
$L_{K(1)}K(\bbFl)\to KU\phat$ (independently of the choice of the
embedding $\bbFlt\to \bC^{\times}$).  Under the identification of
$\pi_{-1}J$ as $H^{1}_{c}(\bZpt;\bZp)$ above, the composite map
\[
J\to L_{K(1)}K(\bZ)\to L_{K(1)}K(\bFl)
\]
induces on $\pi_{-1}$ the map
\[
H^{1}_{c}(\bZpt;\bZp)\to H^{1}(\langle \Phi \rangle;\bZp)
\]
induced by the inclusion of $l$ in $\bZpt$ (the inclusion of
$\Psi^{l}$ in the group of $p$-adically interpolated Adams
operations).  This gives us information about $\alpha$.
In terms of the identification of
$\pi_{-1}(L_{K(1)}K(\bFl))$ as $H^{1}_{\et}(\bFl;\bZp)$ from
Thomason's descent spectral sequence, the map
\[
H^{1}_{\et}(\bFl;\bZp)\iso
H^{1}_{\Gal}(\bbFl/\bFl;\bZp)\to 
H^{1}(\langle \Phi \rangle;\bZp)
\]
is induced by the inclusion of the Frobenius in $\Gal(\bbFl/\bFl)$.
By naturality, the composite map 
\[
H^{1}_{\et}(\bZ[\tfrac1p];\bZp)\iso
H^{1}_{c}(G_{S};\bZp)\to
H^{1}_{c}(\Gal(\bbFl;\bZp))\to
H^{1}(\langle \Phi \rangle;\bZp)
\]
is induced by the inclusion of $\Phi$ in $G_{S}$ as the automorphism
of (the $p$-integers in) $\bQS$ that does the automorphism $\zeta
\mapsto \zeta^{l}$ on $\mu_{p^{\infty}}$.  This then shows that
$\gamma =\alpha$. 

We now compare $\gamma$ and $\eta$. Here it is easiest to work first
in terms of $L_{K(1)}K(\bQp)$.  Using the standard identification of
$\pi_{1}L_{K(1)}K(\bQp)$ as the $p$-completion of the units, we have a
$\bQp$-analogue of $\eta$ using local duality:  Let 
\[
\eta_{p}\colon \pi_{-1}L_{K(1)}K(\bQp)\to \Hom((\bQpt)\phat,\bZp)\iso \Hom(\bQpt,\bZp)
\]
be the isomorphism derived from the isomorphism
$\pi_{1}L_{K(1)}K(\bQp)\iso (\bQpt)\phat$ by Anderson duality.
We then 
have a commutative diagram 
\[
\xymatrix{%
\pi_{-1}L_{K(1)}K(\bZ)\ar[r]^-{\iso}\ar[d]_{\eta}^{\iso}
&\pi_{-1}L_{K(1)}K(\bZ[\tfrac1p])\ar[r]
&\pi_{-1}L_{K(1)}K(\bQp)\ar[d]^{\eta_{p}}_{\iso}\\
\Hom(\bZpt,\bZp)\ar[rr]_-{q^{*}}&&\Hom(\bQpt,\bZp)
}
\]
by the compatibility of local and global duality.  (Here $q^{*}$ is
the map induced by the standard quotient isomorphism $\bQpt/\langle
p\rangle\iso \bZpt$.) To produce a local
analogue of $\gamma$, we use the Artin symbol
$\theta$ in
local class field theory~\cite[\S3.1]{Serre-LocalClassFieldTheory}.
The Artin symbol 
gives an isomorphism between the finite completion of the units of
$\bQp$ and the Galois group of the maximal abelian extension $\bQpab$
of $\bQp$: For $x\in \bQpt$, writing $x=ap^{m}$ for $a\in \bZpt$, the
Artin symbol takes $x$ to the unique element $\theta(x)$ of $\Gal(\bQpab/\bQp)$
that acts on the $p^{n}$th roots of unity $\mu_{p^{n}}$ by raising to
the $1/a$ power and acts on the maximal unramified extension $\bQpnr$
of $\bQp$ by the $m$th power of a lift of the Frobenius.  Using the
isomorphism 
\[
\pi_{-1}L_{K(1)}K(\bQp)\iso H^{1}_{\et}(\bQp;\bZp)
\]
from Thomason's descent spectral sequence and the canonical isomorphism 
\[
H^{1}_{\et}(\bQp;\bZp)\iso
H^{1}_{\Gal}(\bbQp/\bQp;\bZp)\iso \Hom_{c}(\Gal(\bQpab/\bQp),\bZp)
\]
(as above), the Artin symbol induces an isomorphism 
\[
-\gamma_p\colon \pi_{-1}L_{K(1)}K(\bQp)\to \Hom(\bQpt,\bZp).
\]
We have implicitly defined an isomorphism $\gamma_p$: The
formula for the Artin symbol implies that the following diagram commutes
\[
\xymatrix{%
\pi_{-1}L_{K(1)}K(\bZ)\ar[r]^{\iso}\ar[d]_{\gamma}^{\iso}
&\pi_{-1}L_{K(1)}K(\bZ[\tfrac1p])\ar[r]
&\pi_{-1}L_{K(1)}K(\bQp)\ar[d]_{\iso}^{\gamma_p}\\
\Hom(\bZpt,\bZp)\ar[rr]_-{q^{*}}\ar[d]_{=}&&\Hom(\bQpt,\bZp)\ar[d]^{\iso}\\
\Hom(\bZpt,\bZp)\ar[rr]_{(\id,0)}&&\Hom(\bZpt,\bZp)\times \bZp
}
\]
(where the bottom right vertical isomorphism is induced by the $ap^{m}$ decomposition of $\bQpt$ as $\bZpt\times \bZ$).
In other words, omitting notation for the isomorphism arising from
Thomason's descent spectral sequence and the usual isomorphism
$H^{1}_{\et}(\bQp;\bZp)\iso \Hom(\Gal(\bQpab/\bQp),\bZp)$, 
$\gamma_{p}$ is the $\bZp$-dual of $-\theta$. 
The Artin symbol has a cohomological characterization~\cite[\S2.3,
Prop.~1]{Serre-LocalClassFieldTheory}: For a character $\rho\colon
\Gal(\bbQp/\bQp)\to \bQ/\bZ$ and $x\in \bQpt$,
\[
\rho(\theta(x))=-inv(x \cup \rho)
\]
(cf.~\cite[p.~386]{NeukirchSchmidtWinberg})
where on the right we interpret $x$ as an element of
$H^{1}_{\et}(\bQp;\bZp)$ and $\rho$ as an element of
$H^{1}_{\et}(\bQp;\bQ/\bZ)$, while the symbol $\cup$ denotes the cup
product on \'etale cohomology
\[
H^{1}_{\et}(\bQp;\bZp(1))\otimes 
H^{1}_{\et}(\bQp;\bQ/\bZ)\to
H^{2}_{\et}(\bQp;\bQ/\bZ(1)),
\]
and $inv$ denotes the map induced by the Hasse invariant (q.v.~\cite[XIII\S3,
Prop.~6ff]{Serre-LocalFields}).   Because  
\[
inv(x\cup y)\colon H^{1}_{\et}(\bQp;\bZp(1))\otimes
H^{1}_{\et}(\bQp;\bZpi)\to \bQ/\bZ
\]
is the local duality pairing, restricting to 
\[
H^{1}_{\et}(\bQp;\bZ/p^{n}(1))\otimes
H^{1}_{\et}(\bQp;\bZ/p^{n})\to \bZ/p^{n},
\]
taking the inverse limit, and applying the isomorphism from
Thomason's descent spectral sequence gives us the Anderson duality pairing
\[
\pi_{1}L_{K(1)}K(\bQp)\otimes
\pi_{-1}L_{K(1)}K(\bQp)\to \bZp.
\]
We conclude that $\eta_{p}=\gamma_{p}$.  Since the map
$q^{*}\colon \Hom(\bZpt,\bZp)\to \Hom(\bQpt,\bZp)$ is an injection,
we conclude that $\eta=\gamma$.

\section{The eigenspectra of \texorpdfstring{$TC(\bZ)\phat$}{TC(Z)p}}
\label{sec:Zi}

As shown in~\cite[(2.4)]{BM-KSpi}, $TC(\bZ)\phat$ admits a canonical splitting
with summands $j$, $\Sigma j'$, and the spectra denoted in
Section~\ref{sec:review} as $z_{i}$, for $i=0,\dotsc,p-2$.  As
indicated there, $z_{i}$ is non-canonically weakly equivalent to
$\Sigma^{2i-1}\ell$ (for $i\neq 1$) or  $\Sigma^{2p-1}\ell$ (for
$i=1$). The purpose of this section is to give an
identification of these summands in intrinsic terms.  We work in terms
of the $K(1)$-localizations $Z_{i}$, and our main result is to
explain the perspective that $Z_{i}$ is $\Sigma^{2i-1} L$ tensored
over $\Lambda \iso [L,L]$ with a free $\Lambda$-module and to identify
that $\Lambda$-module intrinsically in number theoretic terms; see
Corollary~\ref{cor:DMZto} for a precise statement.  The remainder of
the section discusses the problem of finding a generator for this free
module.  The number theory literature discusses several approaches,
which we review.  The content of this section is independent of the
other sections. 

Let 
\[
Z=Z_{0}\vee \dotsb \vee Z_{p-2}=L_{K(1)}z_{0}\vee\dotsb\vee L_{K(1)}z_{p-2}.
\]
Since
$z_{i}$ is the $1$-connected cover (for $i\neq 0$) or 
$(-2)$-connected cover (for $i=0$) of $Z_{i}=L_{K(1)}z_{i}$, to
identify $z_{i}$, it suffices to identify $Z_{i}$.  Since $Z_{i}$
is non-canonically weakly equivalent to $\Sigma^{2i-1}L$ (for $L$ the
Adams summand of $KU\phat$), we have that $[Z_{i},Z_{j}]=0$ for $i\neq
j$. The decomposition of $Z$ into the summands $Z_{i}$ is therefore unique,
and so it suffices to identify $Z$.  It follows from
Hesselholt-Madsen~\cite[Th.~D, Add~6.2]{HM2}
that $TC$ of the $p$-completion map and the cyclotomic trace
\[
TC(\bZ)\to TC(\bZp)\from K(\bZp)
\]
are $K(1)$-equivalences. These maps induce a weak equivalence from $Z$
to a summand of $L_{K(1)}K(\bZp)$, which
Dwyer-Mitchell~\cite[\S13]{DwyerMitchell} identifies in
terms of units of cyclotomic
extensions of $\bQp$.  We now review this identification.

Let
$F_{n}=\bQp(\mu_{p^{n+1}})$ (with $F_{-1}=\bQp$), and let
$E_{n}=(F_{n}^{\times}/\text{torsion})\phat$ (where $\mu_{p^{n+1}}$
denotes the $p^{n+1}$ roots of unity in some algebraic closure of
$\bQp$).  The norm (in Galois theory) 
gives maps $E_{n}\to E_{n-1}$; let $E_{\infty}$ be the inverse limit.
The Galois group $\Gal(F_{n}/\bQp)$ is canonically isomorphic to
$(\bZ/p^{n+1})^{\times}$ and this makes $E_{n}$ a $p$-complete
$\bZp[(\bZ/p^{n+1})^{\times}]$-module.  The norm $E_{n}\to E_{n-1}$ is a
$\bZp[(\bZ/p^{n+1})^{\times}]$-module map and so in the inverse limit
$E_{\infty}$ is a module over the Iwasawa algebra
\[
\Lambda'=\bZp[[\bZpt]]=\lim \bZp[(\bZ/p^{n+1})^{\times}].
\]
As a completed group ring, $\Lambda'$ comes with an (anti)involution
$\chi$ that sends the group elements to their inverses; for a
$\Lambda'$-module $M$, we denote by $M^{\chi}$ the $\Lambda'$-module
obtained via this involution.
Coincidentally, $\Lambda'\iso [KU\phat,KU\phat]$ via the map that
takes an element $a$ of $\bZpt$ to the ($p$-adically
interpolated) Adams operation $\psi^{a}$ (which exists for
$p$-completed topological $K$-theory), and in particular, we have a
canonical action of $\Lambda'$ on $KU\phat$ (in the stable category).
Dwyer-Mitchell~\cite[\S13]{DwyerMitchell} then shows that
$E_{\infty}$ (denoted there as $E'_{\infty}(\text{red}))$ is free of
rank $1$ as a $\Lambda'$-module, observes that
\[
[Z,KU\phat]=(KU\phat)^{0}(Z)=0
\]
(cf.~[ibid., 6.11]), and constructs a canonical isomorphism 
\begin{equation}\label{eq:Edual}
[\Sigma^{-1}Z,KU\phat]=(KU\phat)^{1}(Z)\iso \Hom_{\Lambda'}(E_{\infty},\Lambda')^{\chi}
\end{equation}
of $\Lambda'$-modules
(cf.~[ibid., 8.10]).
For formal reasons, this characterizes
$Z$ in the stable category via the Yoneda Lemma (see [ibid., 4.12]).

\begin{thm}[Dwyer-Mitchell~{\cite[9.1]{DwyerMitchell}}]\label{thm:DMZ}
For any spectrum $X$, the natural map 
\begin{multline*}
[X,\Sigma^{-1}Z]\to \Hom_{[KU\phat,KU\phat]}([\Sigma^{-1}Z,KU\phat],[X,KU\phat])\\
\iso \Hom_{\Lambda'}(\Hom_{\Lambda'}(E_{\infty},\Lambda')^{\chi},[X,KU\phat])
\iso E_{\infty}^{\chi}\otimes_{\Lambda'}[X,KU\phat]
\end{multline*}
is an isomorphism.
\end{thm}

The last isomorphism follows from the fact that $E_{\infty}$ is free
of finite rank, using the isomorphism 
\[
\Hom_{\Lambda'}(E_{\infty},\Lambda')^{\chi}\iso 
\Hom_{\Lambda'}(E_{\infty}^{\chi},\Lambda')
\]
adjoint to the $\chi$-twisted evaluation map
\begin{equation}\label{eq:chiev}
\Hom_{\Lambda'}(E_{\infty},\Lambda')^{\chi}\otimes E_{\infty}^{\chi}
\to (\Lambda')^{\chi}\xrightarrow[\chi^{-1}]{\iso} \Lambda'.
\end{equation}
Plugging $X=KU\phat$ into Theorem~\ref{thm:DMZ}, we get an isomorphism of
$\Lambda'$-modules 
\[
[KU\phat,\Sigma^{-1}Z]\overto{\iso}
E_{\infty}^{\chi}\otimes_{\Lambda'}[KU\phat,KU\phat]\iso E_{\infty}^{\chi}.
\]
Concisely, this isomorphism and the isomorphism of~\eqref{eq:Edual}
identify the $\chi$-twisted evaluation map~\eqref{eq:chiev} 
with the composition in the stable category 
\[
[\Sigma^{-1}Z,KU\phat]\otimes [KU\phat,\Sigma^{-1}Z]\to
[KU\phat,KU\phat] \iso \Lambda'. 
\]
We have a canonical identification of endomorphism rings $\Lambda'\iso
[KU\phat,KU\phat]$ and 
$[\Sigma^{-1}Z,\Sigma^{-1}Z]$ induced by choosing any weak equivalence
$KU\phat \simeq \Sigma^{-1}Z$: because $\Lambda'$ is
commutative every choice induces the same isomorphism. Commutativity
also implies that the $\Lambda'$-module structure on
$[KU\phat,\Sigma^{-1}Z]$ from $[\Sigma^{-1}Z,\Sigma^{-1}Z]$ coincides with the
$\Lambda'$-structure from $[KU\phat,KU\phat]$ and we see that the
isomorphism in Theorem~\ref{thm:DMZ} is an isomorphism of
$\Lambda'$-modules for the $\Lambda'$-module structure on
$[X,\Sigma^{-1}Z]$ from  $[\Sigma^{-1}Z,\Sigma^{-1}Z]$.  Using the duality of the invertible
$\Lambda'$-modules $[\Sigma^{-1}Z,KU\phat]$ and
$[KU\phat,\Sigma^{-1}Z]$, Theorem~\ref{thm:DMZ} then implies the following
slightly less complicated isomorphism.

\begin{cor}
For any spectrum $X$, the natural map
\[
E_{\infty}^{\chi}\otimes_{\Lambda'}[X,KU\phat] \iso
[KU\phat,\Sigma^{-1}Z]\otimes_{[KU\phat,KU\phat]} [X,KU\phat]
\to [X,\Sigma^{-1}Z]
\]
is an isomorphism.
\end{cor}

We can identify the spectra $Z_{i}$ using the eigensplitting of
$E_{\infty}^{\chi}$.  To explain this, let $\Lambda$ be the completed group ring
\[
\Lambda =\bZp[[U^{1}]]=\lim \bZp[U^{1}/U^{n}]
\]
where $U^{n}$ denotes the subgroup of $\bZpt$ of elements congruent to
$1$ mod $p^{n}$.  Then $\Lambda$ is a (topological) subring of
$\Lambda'$ and because $\bZpt\iso\mu_{p-1}\times U^{1}$, 
$\Lambda'$ is isomorphic to the group algebra $\Lambda [\mu_{p-1}]$.
We prefer to write $\Lambda'\iso \Lambda [\bF_{p}^{\times}]$, using
the Teichm\"uller character $\omega \colon \bF_{p}^{\times}\to \bZpt$
for the isomorphism $\bF_{p}^{\times}\to \mu_{p-1}$.  Because $p-1$ is
invertible in $\bZp$ and $\mu_{p-1}\subset \bZp$, we get orthogonal
idempotent elements 
\[
\epsilon_{i}=\frac1{p-1}\sum_{\alpha \in \bF_{p}^{\times}}\omega(\alpha)^{-i} \psi^{\omega(\alpha)}
\]
in $\Lambda'$ for $i=0,\dotsc,p-2$ (where we have written the group
elements (in $\bZpt$) using Adams operation notation to distinguish
them from the coefficient elements (in $\bZp$) of the completed group
algebra $\Lambda'$).  These give a Cartesian product decomposition of 
$\Lambda'$, 
\[
\Lambda'=\epsilon_{0}\Lambda'\times \dotsb \times \epsilon_{p-2}\Lambda'
\]
where the elements of $\epsilon_{i}\Lambda$ can be characterized as
the elements of $\Lambda'$ on which $\bF_{p}^{\times}$ acts by
multiplication by $\omega^{i}$; we call this the $\omega^{i}$
eigenspace of $\Lambda'$.  Since $\omega^{i}=\omega^{i+p-1}$, it makes
sense and can be convenient to index the eigenspaces on elements of
$\bZ/(p-1)$ rather than the specific representatives $0,\dotsc,p-2$.  
The inclusion $\Lambda \to \Lambda'$
induces an isomorphism of (topological) $\bZp$-algebras between
$\Lambda$ and $\epsilon_{i}\Lambda'$ for each $i$.  Every
$\Lambda'$-module admits a corresponding decomposition into
$\omega^{i}$ eigenspaces, which we can regard as $\Lambda$-modules.

The Cartesian product decomposition of $\Lambda'$ above corresponds
exactly to the Cartesian product decomposition of $[KU\phat,KU\phat]$, 
\[
[KU\phat,KU\phat]\iso [L,L]\times [\Sigma^{2}L,\Sigma^{2}L]\times \dotsb \times [\Sigma^{2p-4}L,\Sigma^{2p-4}L]
\]
induced by the Adams splitting $KU\phat \simeq L\vee
\Sigma^{2}L\vee\dotsb\vee\Sigma^{2p-4}L$: The decomposition
isomorphism takes $[\Sigma^{2i}L,\Sigma^{2i}L]$ to precisely the
subset of $[KU\phat,KU\phat]$ of elements on which
$\psi^{\omega(\alpha)}$ acts by multiplication by $\omega(\alpha)^{i}$
for all $\alpha\in \bF_{p}^{\times}$.  The isomorphism $\Lambda'\iso
[KU\phat,KU\phat]$ induces an isomorphism $\Lambda \iso [L,L]$ for the
inclusion of $[L,L]$ as the diagonal in $\prod [L,L]\iso \prod
[\Sigma^{2i}L,\Sigma^{2i}L]$.  Since the identification of
$[KU\phat,\Sigma^{-1}Z]$ as $E_{\infty}^{\chi}$ above is an isomorphism of
$\Lambda'$-modules, we get an isomorphism of $\Lambda$-modules
\[
\epsilon_{i}[KU\phat,\Sigma^{-1}Z]\iso \epsilon_{i}E_{\infty}^{\chi}
\]
on the $\omega^{i}$ eigenspaces for all $i$.  The idempotent
$\epsilon_{i}$ of $[KU\phat,KU\phat]$ is the composite of the
projection and inclusion
\[
KU\phat \to \Sigma^{2i}L\to KU\phat,
\]
so we see that 
\[
\epsilon_{i}[KU\phat,\Sigma^{-1}Z]\iso 
[\Sigma^{2i}L,\Sigma^{-1}Z]\iso
[\Sigma^{2i}L,\Sigma^{-1}Z_{i+1}]
\]
(numbering mod ${p-1}$) as $[\Sigma^{2i}L,\Sigma^{-1}Z_{j}]=0$ unless
$j\equiv i+1\pmod{p-1}$.  We then get the following characterization
of $Z_{i}$.

\begin{cor}\label{cor:DMZto}
Let $i\in \bZ/(p-1)$.  For any spectrum $X$, the natural map
\[
\epsilon_{i} E_{\infty}^{\chi}\otimes_{\Lambda}[X,\Sigma^{2i}L] \iso
[\Sigma^{2i}L,\Sigma^{-1}Z_{i+1}]\otimes_{[\Sigma^{2i}L,\Sigma^{2i}L]} [X,L]
\to [X,\Sigma^{-1}Z_{i+1}]
\]
is an isomorphism.
\end{cor}

More can be said about choosing weak equivalences $Z_{i}\simeq
\Sigma^{2i-1}L$.  From the discussion above, choosing such a weak
equivalence is equivalent to choosing a
generator of the free rank one $\Lambda$-module
$\epsilon_{i-1}E_{\infty}^{\chi}$, and this is equivalent to choosing a
generator of the free rank one $\Lambda$-module $\epsilon_{p-i}E_{\infty}$. 

We treat the case of $\epsilon_{0}E_{\infty}$ separately, but all
cases require a choice of a system of primitive $p$th roots of units
$\zeta_{n}\in \mu_{p^{n+1}}$ with $\zeta_{n}^{p}=\zeta_{n-1}$, which
we now fix. (For example, choosing an embedding of $\bQp$ in $\bC$,
one could take $\zeta_{n}=e^{2\pi i/p^{n+1}}$.)  Then $\zeta_{n}-1$ is
a uniformizer for $\NR_{F_{n}}$ for each $n$: we have
$N_{F_{n+1}/F_{n}}(\zeta_{n+1}-1)=\zeta_{n}-1$ and
$N_{F_{0}/\bQ_{p}}(\zeta_{0}-1)=p$. This argument also shows that the
system $(\!(\zeta_{n}-1)\!)$ specifies an element of $E_{\infty}$,
and we can consider its $\omega^{0}$
eigenfactor $\epsilon_{0}(\!(\zeta_{n}-1)\!)$.

\begin{prop}\label{prop:zerogen}
In the $\Lambda'$-module $E_{\infty}$, $\epsilon_{0}(\!(\zeta_{n}-1)\!)$ generates $\epsilon_{0}E_{\infty}$
\end{prop}

\begin{proof}
The valuation $F_{n}^{\times}\to \bZ$ is a homomorphism that sends
roots of unity to zero and so extends to a homomorphism $E_{n}\to
\bZp$.  It commutes with the norm and so assembles to a homomorphism
$E_{\infty}\to \bZp$.  Giving $\bZp$ the trivial $\Lambda'$-action,
the homomorphism $E_{\infty}\to \bZp$
is $\Lambda'$-linear and factors through $\epsilon_{0}E_{\infty}$.
Because $\zeta_{n}-1\in F_{n}^{\times}$ has valuation $1$, $\Lambda$ is a local ring,
and $\epsilon_{0}E_{\infty}$ is a free rank one $\Lambda$-module, it
follows by Nakayama's Lemma
that $\epsilon_{0}(\!(\zeta_{n}-1)\!)$ is a generator of $\epsilon_{0}E_{\infty}$.
\end{proof}

\begin{rem}\label{rem:Y1}
We note that the element $\epsilon_{0}(\!(\zeta_{n}-1)\!)$ of
$\epsilon_{0}E_{\infty}$ is in the image of the corresponding
$\Lambda'$-module
defined in terms of $\bZ[1/p]$ in place of $\bQp$, that is, the
inverse limit over norm maps of the $p$-completion of the units of
$\bZ[\zeta_{n},1/p]$ modulo torsion. This element can be
used to construct a weak equivalence $\Sigma^{-1}Y_{1}\simeq L$ by an argument
analogous to the one for Theorem~\ref{thm:DMZ} (but for the
$\epsilon_{0}$ piece only).  Since
the resulting weak equivalence is just the composite of the weak
equivalence $\Sigma^{-1}Y_{1}\to \Sigma^{-1}Z_{1}$ induced by the
cyclotomic trace and the weak equivalence $\Sigma^{-1}Z_{1}\simeq L$
coming from the previous proposition (and Corollary~\ref{cor:DMZto}),
we omit the details.
\end{rem}

For other eigenspaces, $\epsilon_{i}(\!(\zeta_{n}-1)\!)$ is generally not a
generator.  If $i\not\equiv 0$ mod $p-1$,
then $\epsilon_{i}(\zeta_{n}-1)\in F_{n}^{\times}$ is a
\term{cyclotomic unit}~\cite[3\S5]{Lang-CyclotomicIandII}. In
particular, it is a real multiple of a root of
unity~\cite[p.~84]{Lang-CyclotomicIandII}, and as such,
$\epsilon_{i}(\zeta_{n}-1)$ becomes the 
identity element in $\epsilon_{i}E_{n}$ for $i$ odd.  For $i$
even, $\epsilon_{i}(\!(\zeta_{n}-1)\!)$ is a generator of
$\epsilon_{i}E_{\infty}$ if and only if the
Bernoulli number $B_{i}$ is relatively prime to $p$ (see the argument
for Theorem~1.4 in Chapter~7 of Lang~\cite{Lang-CyclotomicIandII}).

The authors know of two (for $i=1$) or three (for $i=2,\dotsc,p-2$)
distinct ways of producing generators for the other eigenspaces.  For
the constructions, let 
$U^{1}_{n}$ denote the subgroup of $F^{\times}_{n}$ congruent to $1$ mod
$\zeta_{n}-1$ and let $U^{1}_{\infty}=\lim U^{1}_{n}$ under the norm
maps. Since $U^{1}_{n}$ is $p$-complete, the Galois action on $U^{1}_{\infty}$
makes it a $\Lambda'$-module.  Let $T_{p}(\mu)=\lim \mu_{p^{n+1}}$,
a $\Lambda'$-submodule of $U^{1}_{\infty}$.  Since $\mu_{p^{n+1}}$ is the
torsion of $U^{1}_{n}$ and $\mu_{(p-1)p^{n+1}}$ is the torsion in
$F^{\times}$, we have an exact sequence of $\Lambda'$-modules 
\[
0\to T_{p}(\mu)\to U^{1}_{\infty}\to E_{\infty}\to \bZp\to 0
\]
where $\bZp$ has the trivial Galois action and the map $E_{\infty}\to
\bZp$ is induced by the valuation as in the proof of
Proposition~\ref{prop:zerogen} above.  In particular
$U^{1}_{\infty}\to E_{\infty}$ is an isomorphism on $\omega^{i}$
eigenspaces for $i\neq 0,1$ and an epimorphism on $\omega^{1}$. 

Given an element $(\!(u_{n})\!)\in U^{1}_{\infty}$, we can detect whether
$\epsilon_{i}(\!(u_{n})\!)$ maps to a generator of $\epsilon_{i}E_{\infty}$
using the \term{Kummer homomorphisms} (see
\cite[7\S1--2]{Lang-CyclotomicIandII}) $\phi_{i}\colon U^{1}_{0}\to
\bF_{p}$. Let $D\colon \bZp[[X]]\to \bZp[[X]]$ be the homomorphism of
$\bZp$-modules $Df=(1+X)f'(X)$.  Given $u\in U_{0}^{1}$,
$u=f_{u}(\zeta_{0}-1)$ for some (non-unique) $f_{u}\in \bZp[[X]]$ with
leading coefficient congruent to $1$ mod $p$; for $i=1,\ldots,p-2$,
define  
\[
\phi_{i}(u)=D^{i}(\log(f_{u}))|_{X=0} \pmod p \in \bF_{p}
\]
where $\log(f)=(f-1)-(f-1)^{2}/2+\cdots$ (or more generally, for any
power series $f$ with constant term a unit in $\bZp$, we can equivalently
interpret $D^{i}\log(f)$ as $D^{i-1}((1+X)f'/f)$).
The power series $f_{u}$ is well defined mod $(p,X^{p-1})$ since
$\NR_{F_{0}}$ has ramification index $p-1$ over $\bZp$.  As a consequence,
$f_{u}(X)f_{v}(X)\equiv f_{uv}(X)\mod{(p,X^{p-1})}$.  The formal
power series identity for $\log(1+X)$ (or the product rule applied to
$(gh)'/(gh)$) implies 
\[
D^{i}\log(f_{u}f_{v})=D^{i}\log(f_{u})+D^{i}\log(f_{v})
\]
and $\phi_{i}$ is a well-defined homomorphism.  The (easily checked)
formula
\[
D(f((X+1)^{a}-1))=a(X+1)^{a}f'((X+1)^{a}-1)=aDf|_{(X+1)^{a}-1}
\]
shows that
$\phi_{i}(\psi^{a}u)=a^{i}\phi_{i}(u)$ for any $u\in U_{0}^{1}$, $a\in
\bZpt$.  It follows that $\phi_{i}(\epsilon_{i}u)=\phi_{i}(u)$. Since
$\epsilon_{i}E_{\infty}$ is a free rank one $\Lambda$-module and $\Lambda$ is a
local ring, if $\phi_{i}(u_{0})\neq 0$ for some
$(\!(u_{n})\!)\in U_{\infty}^{1}$, and either $i\in \{2,\ldots,p-2\}$ or $i=1$ and
$\epsilon_{1}u_{0}\notin \mu_{p}$, then the image of $\epsilon_{i}(\!(u_{n})\!)$ generates
$\epsilon_{i}E_{\infty}$.  For the first construction of generators,
we use the following proposition
from~\cite[7\S3]{Lang-CyclotomicIandII}.

\begin{prop}\label{prop:LangUnits}
Fix $i\in 1,\ldots,p-2$.  There exists $\lambda \in \mu_{p-1}, \lambda
\neq 1$ such that $\phi_{i}(\omega(\lambda -1)^{-1}(\lambda -\zeta_{0}))\neq 0$.
\end{prop}

The proof is an analysis of $\phi_{i}(u_{0}(\lambda))$ for
$u_{0}(\lambda)=\omega(\lambda -1)^{-1}(\lambda -\zeta_{0})$.
Clearly, we can take 
$f_{u_{0}(\lambda)}=\omega(\lambda -1)^{-1}(\lambda -1-X)$.
We then have $D(\log(f_{u_{0}}))|_{X=0}=1/(1-\lambda)$, and in general, 
a little bit of algebra (q.v.~\cite[p.~182]{Lang-CyclotomicIandII})
shows that $D^{i}(\log f_{u_{0}(\lambda)})|_{X=0}$ is a
rational polynomial in $\lambda$ whose numerator is a monic polynomial
of degree $i-1<p-2$; thus, for at least one of the $p-2$ values of
$\lambda\in \mu_{p-1}\setminus\{1\}$, $\phi_{i}(u_{0}(\lambda))$ must
be non-zero mod $p$.  

For the $\omega^{1}$ eigenspace, we also need the following observation.

\begin{prop}
Let $\lambda \in\mu_{p-1}\setminus\{1\}$ and let $u_{0}(\lambda)=\omega(\lambda
-1)^{-1}(\lambda -\zeta_{0})$.  Then $\epsilon_{1}(u_{0}(\lambda))$ is not a pth
root of unity. 
\end{prop}

\begin{proof}
Let $\pi=\zeta_{0}-1\in \NR_{F_{0}}$; we note $(p)=(\pi^{p-1})$.
It suffices to check that 
\[
\epsilon_{1}(u_{0}(\lambda))^{p}\not\equiv
1\pmod{\pi^{p+1}}.
\]
Since $\omega(\lambda -1)^{-1}(\lambda-1)$ is in $\bZp$, and is by
construction congruent to $1$ mod $p$ (\textit{a fortiori} congruent
to 1 mod $\pi$ in $\NR_{F_{0}}$), it is in $\epsilon_{0}U^{1}_{0}$ and $\epsilon_{1}$
takes it to the identity.  We then have
\[
\epsilon_{1}(u_{0}(\lambda))=
\epsilon_{1}\left(\frac{u_{0}(\lambda)}{\omega(\lambda -1)^{-1}(\lambda-1)}\right)
=\epsilon_{1}\left(1-\frac{\pi}{\lambda -1}\right).
\]
and 
\[
\def\mystrut{\vbox to 1.25em{\vss}}
\left(1-\frac{\pi}{\lambda -1}\right)\mystrut^{p}\equiv
1-\frac{p\pi}{\lambda -1}\pmod {\pi^{p+1}}. 
\]
Moreover, since
\[
\def\mystrut{\vbox to 1.25em{\vss}}
\psi^{a}\left(1-\frac{p\pi}{\lambda -1}\right)\equiv
1-\frac{ap\pi}{\lambda -1}
\equiv
\left(1-\frac{p\pi}{\lambda -1}\right)\mystrut^{a}
\pmod {\pi^{p+1}},
\]
it follows that
$\epsilon_{1}(u_{0}(\lambda))^{p}\equiv 1-p\pi/(1-\lambda)\not\equiv 1\pmod{\pi^{p+1}}$.
\end{proof}

Taking $\lambda$ as in Proposition~\ref{prop:LangUnits}, let
$u_{n}=\omega(\lambda -1)^{-1}(\lambda -\zeta_{n})$.  Because $\lambda
\in \mu_{p-1}$, we have $N_{F_{n+1}/F_{n}}(u_{n+1})=u_{n}$, and
the system $\epsilon_{i}(\!(u_{n})\!)$ maps to a generator of
$\epsilon_{i}E_{\infty}$.

\begin{prop}
Let $i\in 1,\dotsc,p-2$, and choose
$\lambda \in \mu_{p-1}\setminus\{1\}$ such that\break
$\phi_{i}(\omega(\lambda -1)^{-1}(\lambda -\zeta_{0}))\neq 0$.  Then
$\epsilon_{i}(\omega(\lambda -1)^{-1}(\lambda -\zeta_{n}))\in
\epsilon_{i}U^{1}_{\infty}$ maps to a generator of $\epsilon_{i}E_{\infty}$.
\end{prop}

The next construction of generators, due to
Coates-Wiles~\cite[Theorem~4]{CoatesWiles-EllipticUnits}, avoids the
indefiniteness of the unspecified choice $\lambda$ in the previous
construction.  Let  
$\beta\in \bZp$ denote the unique $(p-1)$th root of $1-p$ which is
congruent to $1$ mod $p$.  Let $\Gamma$ be the unique Lubin-Tate formal
group law over $\bZp$ with $[p]_{\Gamma}[X]=X^{p}+pX$, and let
$\theta$ be the unique strict isomorphism from the multiplicative
formal group law $G_{m}$ to $\Gamma$.  (For an introduction to this
Lubin-Tate theory, see for example~\cite[8\S1]{Lang-CyclotomicIandII},
particularly Theorems~1.1 and~1.2.)  Because $\pi_{n}=\zeta_{n}-1$
satisfies 
\[
[p]_{G_{m}}(\pi_{n+1})=(\pi_{n+1}+1)^{p}-1=\pi_{n},
\]
for $n\geq 0$, $x_{n}=\theta(\pi_{n})$ satisfies
\[
x_{n+1}^{p}+px_{n+1}=x_{n}
\]
for $n\geq 0$, and $x_{0}^{p}+px_{0}=0$.  Let $u_{n}=\beta-x_{n}$.
We have $N_{F_{n+1}/F_{n}}(u_{n+1})=u_{n}$, and it is easy to
calculate $\phi_{i}(u_{0})$ as follows.  Because $\Gamma$ agrees with
the additive formal group law to order $p-1$, $\theta(X)\equiv
\log(1+X)\pmod {X^{p}}$.  Calculating
\[
D^{i}(\beta - \log(1+X))=\frac{-(i-1)!}{(\beta -\log(1+X))^{i}},
\]
we get $\phi_{i}(u_{0})=-(i-1)!\neq 0 \in \bF_{p}$.
In the case $i=1$, $\epsilon_{1}u_{0}$ is not a $p$th root of unity
(see~\cite[2.5]{Saikia-Colman}).  This proves that
$\epsilon_{i}(\!(u_{n})\!)$ maps to a generator of $\epsilon_{i}E_{\infty}$
for all $i=1,\ldots,p-2$.

\begin{prop}
Let $i\in 1,\dotsc,p-2$.  In the notation of the previous paragraph,
$\epsilon_{i}(\!(\beta -\theta(\zeta_{n}-1))\!)\in
\epsilon_{i}U^{1}_{\infty}$ maps to a generator of
$\epsilon_{i}E_{\infty}$.
\end{prop}

Finally, 
Coleman~\cite{Coleman-DivisionValues,Coleman-LubinTateTowers}
constructs an extremely nice isomorphism $\epsilon_{i}E_{\infty}\to
\Lambda$ for $i\neq 0,1$.  
The Coleman map $\oL\colon U^{1}_{\infty}\to \Lambda'$ fits into an exact
sequence of $\Lambda'$-modules
\[
0\to T_{p}(\mu)\to U^{1}_{\infty}\overto{\oL}\Lambda'\to \bZp(1)\to 0
\]
(see~\cite[3.5.1]{CoatesSujatha} or \cite[Theorem~2]{Saikia-Colman}),
where $\bZp(1)$ is the $\Lambda'$-module with underlying $\bZp$-module
$\bZp$ where $\bZpt$ acts by
multiplication (it is isomorphic to $T_{p}(\mu)$ under the isomorphism
$a\mapsto (\!(\zeta_{n}^{a})\!)$).  In particular, the Coleman map is an
isomorphism on $\omega^{i}$ eigenspaces for $i\neq 1$.  Choosing a
generator of $\epsilon_{i}\Lambda'$ for $i\neq 0,1$, $\oL^{-1}$ gives
a generator of $\epsilon_{i}U^{1}_{\infty}$ which maps to a generator
of $\epsilon_{i}E_{\infty}$.  (This generator has $\phi_{i}$ equal to
$1$ by~\cite[3.5.2]{CoatesSujatha}.)


\bibliographystyle{plain}
\bibliography{bluman}

\end{document}